\documentclass[12pt, a4paper]{amsart}

\usepackage{amsmath,amssymb,amsthm}
\usepackage{hyperref}
\usepackage[dvipsnames,usenames]{color}

\newtheorem{guia}{}

\newtheorem{propo}[guia]{Proposition}
\newtheorem{teorema}[guia]{Theorem}

\newtheorem{lema}[guia]{Lemma}

\newcommand{\ds}{\displaystyle}

\newcommand{\e}{\varepsilon}

\newcommand{\R}{\mathbb R}

\newcommand{\re}[1]{(\ref{#1})}
\newcommand{\begeqa}{\begin{eqnarray}}
\newcommand{\eneqa}{\end{eqnarray}}
\newcommand{\begeqaet}{\begin{eqnarray*}}
\newcommand{\eneqaet}{\end{eqnarray*}}
\newcommand{\beeq}{\begin{equation}}
\newcommand{\eeq}{\end{equation}}
\newcommand{\beeqs}{\begin{equation*}}
\newcommand{\eeqs}{\end{equation*}}

\begin{document}

\title[Solvability of nonlinear elliptic equations]{Solvability of nonlinear elliptic equations with gradient terms}

\author[P. Felmer]
{ Patricio Felmer}
\author[A. Quaas]
{ Alexander Quaas}
\author[B. Sirakov]
{Boyan Sirakov }
\
\date{}
\address{P. Felmer \hfill\break\indent
Departamento de Ingenier\'{\i}a  Matem\'atica and
Centro de Modelamiento Matem\'a\-tico
UMR2071 CNRS-UChile,
Universidad de Chile\\
Casilla 170 Correo 3, Santiago, Chile.
} \email{{\tt pfelmer@dim.uchile.cl}}

\address{A. Quaas\hfill\break\indent
Departamento de Matem\'{a}tica, Universidad T\'ecnica Federico Santa Mar\'{\i}a
\hfill\break\indent  Casilla V-110, Avda. Espa\~na, 1680 --
Valpara\'{\i}so, CHILE.}
\email{{\tt  alexander.quaas@usm.cl}}

\address{B. Sirakov\hfill\break\indent
Pontificia Universidade Cat\'olica do Rio de Janeiro (PUC-Rio)\\
Departamento de Matem\'atica\\
Rua Marques de S\~{a}o Vicente, 225, G\'avea\\
Rio de Janeiro - RJ\\
CEP 22451-900, Brasil.
}\email{{\tt bsirakov@mat.puc-rio.br}}


\begin{abstract}

We study the solvability in the whole Euclidean space of coercive quasi-linear and fully nonlinear elliptic equations modeled on $\Delta u\pm g(|\nabla u|)= f(u)$, $u\ge0$, where $f$ and $g$ are increasing continuous functions. We give conditions on $f$ and $g$ which guarantee the availability or the absence of positive solutions of such equations in $\R^N$. Our results considerably improve the existing ones and are sharp or close to sharp in the model cases. In particular, we completely characterize the solvability of such equations when  $f$ and $g$ have  power growth at infinity. We also derive a solvability statement for coercive equations in general form.
\end{abstract}

\maketitle

\section{Introduction}
\setcounter{section}{1}
\setcounter{equation}{0}

A topical problem in the theory of second-order elliptic PDE is the availability  of positive solutions in the whole Euclidean space of equations depending nonlinearly both in the unknown function and its gradient, such as
\beeq\label{princ}
Q[u]= H(u, |\nabla u|)\quad \mbox{in }\quad \R^N,\; N\ge 2,\\
\eeq
where $Q$ is a linear, quasi-linear or fully nonlinear second order elliptic operator and $H$ is a continuous function. In this paper we consider the case when the right hand side in \re{princ} compares to sums or differences of  nonlinear functions in $u$ and $|\nabla u|$. The model cases for our study, for which all our results are new, are the equations
$$
\begin{array}{ll}
\Delta u= f(u)\pm g(|\nabla u|)& \mbox{in }\R^N,\\
\end{array}
\leqno (P_\pm)
$$
where
\beeq\label{condfg}
 f,g\in C([0,\infty))\;\mbox{ are increasing functions, with }\; f(0)=g(0)=0.
\eeq
 We deduce existence results for much more general nonlinearities and second order operators, by applying comparison principles. We will  consider autonomous equations only, for the sake of readability and conciseness.
 \medskip

Throughout the paper, an existence statement will mean that \re{princ} has a positive {\it classical} solution  such that $u(x)\to\infty$ as $|x|\to \infty$, whereas a non-existence statement will mean \re{princ} does not have even {\it weak} (viscosity) solutions, without any assumption on their behavior at infinity.
\medskip

The importance of equations like  $(P_\pm)$ has been known since the work by Lasry and Lions \cite{LL} where they studied solutions defined on a bounded domain, and exploding on the boundary of the domain. Through Bellman's programming principle the solution of \re{princ} can be viewed as the value function of a stochastic control process, and the boundary condition then means that the process is discouraged to leave the domain by setting an infinite cost on the boundary. Here we will consider the natural situation when there is no restriction on the process and it is allowed to evolve in the whole space, but the cost increases as we move away from the starting point. As is well-known, the viscosity solutions framework is adapted to the study of optimal control problems.

Following \cite{LL} there has been a huge number of works on explosive solutions in bounded domains, and it is outside the scope of this paper to even attempt to give a full bibliography. A good starting point to the literature could be the recent survey  \cite{Ra}. Another large set of references can be found in the recent work \cite{AGQ}. That paper contains the most general conditions to date on existence and non-existence for the particular problems  $(P_\pm)$ in bounded domains, and we will return to it below.

 On the other hand, to our knowledge there has only been a limited number of studies on gradient-dependent equations of the type \re{princ} in unbounded domains. Among these, in  \cite{FS} Farina and Serrin have performed a very complete and general study of such equations provided $H$ in \re{princ}  behaves like a product of a term in $u$ and a term in $|\nabla u|$, with power growth. Some partial results on nonexistence for  problems such as $(P_\pm)$ were obtained in the work of Filippucci, Pucci and Rigoli, where the gradient term is also assumed to have limited power growth. Existence and non-existence results for problems such as $\Delta u= \rho(x)f(u)\pm |\nabla u|^q$ with a weight $\rho(x)$ which is assumed to have suitable decay at infinity can be found in the works by Lair and Wood \cite{LW}, and  Ghergu, Niculescu and Radulescu \cite{GNR}.  For a deep and extensive study of the validity of various forms of the maximum principle for equations like \re{princ} we refer to the recent book by Pucci and Serrin \cite{PS}. The interested reader may consult these references for more context.

The most important conceptual novelty of our work is that we give a rather precise description of the way the \textit{interaction} between $f$ and $g$ influences the solvability of $(P_\pm)$. As a first simple consequence,
 we are able to completely characterize the solvability of $(P_\pm)$ in the case when $f$ and $g$ have power-growth behavior at infinity, a question which has been open for some time. Furthermore, since our results are independent of the specific form of $f$ and $g$, they easily apply to  more general nonlinearities than power functions. For instance, for $(P_+)$ we can completely characterize all power-log functions $f$ and $g$ for which this problem has a positive solution (we observe however that we fall short of full characterization for such functions in the case of $ (P_-)$).
\medskip

 Before starting our discussion on problems with gradient dependance, we recall the very classical result by  Keller and Osserman (see \cite{O}, \cite{Ke}, and  also \cite{Ra}, \cite{Fa}, \cite{DGGW}), which states that
the equation $(P_\pm)$ without a gradient term, that is,
\beeq\label{KOeq}
\Delta u = f(u)\quad \mbox{in }\;\R^N
\eeq
has a  positive solution  if and only if
\begin{equation}\label{KO}
\int_1^{\infty}\frac{ds}{\sqrt{F(s)}}=\infty,
\end{equation}
where $F(t)=\int_0^t f(s)ds$ is the primitive of $f$. Standard examples of such $f(t)$ are functions whose growth at infinity does not exceed that of $t^p$, $p\le1$, or of $t(\log t)^q, q\le 2$, or $t(\log t)^2 (\log\log t)^q, q\le2$, etc.
\medskip

We now move to the problem $(P_+)$. As it is  easy to see (and will be checked below), the problem without explicit $u$-dependence
\beeq\label{gKOeq}
\Delta u = g(|\nabla u|)\quad \mbox{in }\;\R^N
\eeq
 has a positive solution if and only if
\begin{equation}\label{gKO}
\int_1^{\infty}\frac{ds}{g(s)}=\infty.
\end{equation}
Examples of such $g(t)$ include functions which grow at infinity at most as  $t^p$, $p\le1$, or $t(\log t)^q, q\le 1$, etc.

Roughly speaking, the point of conditions \re{KO} and \re{gKO} is that they describe quantitatively how quickly at most $f(s)$ or $g(s)$ can grow as $s\to \infty$ so the nonlinear terms in \re{KOeq} and \re{gKOeq} do not force the second derivative to become so large that a  blow-up at a finite point occurs. Note that these equations are {\it coercive}, in the sense that an increase in the unknown function or its derivative leads to an increase in the second order operator.

It is an obvious, and yet still open, question whether \re{KO} and \re{gKO} are necessary and sufficient for the full problem $(P_+)$ to have a positive solution, or on the contrary, the terms in $u$ and $|\nabla u|$ can somehow cooperate in order to prevent solutions from existing globally, while \re{KO} and \re{gKO} are both satisfied. Our first theorem gives evidence against the latter, and in particular implies that it does not occur for any standard choice of $f$ and $g$ satisfying \re{KO} and \re{gKO}.

\begin{teorema}\label{teo32}Let $f$ and $g$ be functions satisfying \re{condfg}. Then

{\rm (i)} If at least one of the assumptions \re{KO} and \re{gKO} is not satisfied then
any nonnegative subsolution to $(P_+)$ vanishes identically.

{\rm (ii)} If \re{KO} and \re{gKO} are satisfied, and, in addition, there are numbers  $A_0, \epsilon_0>0$  such that for all $A\ge A_0$ either
\beeq\label{maincond}\limsup_{s\to +\infty}\frac{\ g(A\sqrt{F(s)})}{A^2 f(s)}<\frac{1}{2}-\epsilon_0 \qquad \mbox{or}\qquad
\liminf_{s\to +\infty}\frac{ g(A\sqrt{F(s)})}{A^2 f(s)}> \frac{1}{2}+\epsilon_0,
\eeq
then $(P_+)$ admits at least one positive solution.
\end{teorema}

A discussion of condition \re{maincond} is in order. This hypothesis essentially requires that the functions $g\circ \sqrt{F}$ and $f$  be comparable for large values of their argument. Obviously if \re{KO} and \re{gKO} are true then \re{maincond} can fail only for handcrafted examples of $f$ and $g$, in which $f$ is built dependently on $g$ or vice versa. On the contrary, \re{maincond} is verified for any standard examples of functions $f$ and $g$ which satisfy \re{KO} and \re{gKO}. Roughly speaking, if $g\circ \sqrt{F}$ grows no faster than $f$ at infinity, then either the limsup in \re{maincond} is zero, or it grows at most as  $g(A)/A^2$; but $g(A)$ grows more slowly than $A^2$ as $A\to\infty$, thanks to \re{gKO}; while if $g\circ \sqrt{F}$ grows strictly faster than $f$ at infinity, then the liminf in \re{maincond} is infinite. Note that $g(ab)\le C g(a)g(b)$ as $a,b\to \infty$  is  verified by all standard examples for~\re{gKO}.

As a strightforward application of Theorem \ref{teo32} and the preceding remark we see that in the simplest case when $f$ and $g$ have power growth at infinity, we can give a complete picture of the solvability of $(P_+)$. Specifically,
 if $f(s)$ behaves like $s^p$ and $g(s)$ behaves like $s^q$ as $s\to \infty$, for some $p,q>0$  then a necessary and sufficient condition for  $(P_+)$ to have  a positive solution is
 \begin{equation}\label{condi1}
 \max\{ p,q \} \le 1.
 \end{equation}
 Similarly, and more generally, if $f$ and $g$  grow at infinity like $s^p(\log s)^\alpha$, $s^q(\log s)^\beta$, it is an easy exercise to plug these functions in \re{KO}, \re{gKO} and \re{maincond}, and completely characterize the solvability of $(P_+)$ in terms of $p,q,\alpha,\beta$, etc.
\medskip

\noindent {\it Remark 1}. As we shall see later, it is not difficult to remove \re{maincond} in the particular cases when $g$ has at most linear growth at infinity or $f$ does not grow faster than $g$ at infinity, in the following sense
\beeq\label{supplcond}
\limsup_{s\to\infty} \frac{g(s)}{s}<\infty \qquad \mbox{or}\qquad \limsup_{s\to\infty} \frac{f(as)}{g(s)}<\infty\quad \forall\:a>0.
\eeq
\medskip

 Next we turn to the problem $(P_-)$. We introduce the following function
\begin{equation}\label{def-Gamma}
\Gamma(s)=\int_{0}^{2s}g(t)\,dt+2Ns^2.
\end{equation}
In what follows we will actually be interested in the behavior at infinity of the inverse function of $\Gamma$, which by \re{def-Gamma} grows like $\sqrt{s}$ if $g$ has at most linear growth, while $\Gamma^{-1}(s)$ has strictly slower growth than $\sqrt{s}$ if $g$ is superlinear at infinity.

\begin{teorema}\label{teo22}
Let $f$ and $g$ be functions satisfying \re{condfg}. Then

{\rm (i)} If \begin{equation}\label{intro-cond2}
\int_1^{\infty}\frac{ds}{\Gamma^{-1}\big(F(s)\big)}<\infty,
\end{equation}
 then any nonnegative subsolution to $(P_-)$ vanishes identically.

{\rm (ii)} If
\begin{equation}\label{intro-cond1}
\int_1^{\infty}\frac{ds}{\sqrt{F(s)}}=\infty\qquad\mbox{or}\qquad
\int_1^{\infty}\frac{ds}{g^{-1}\big(f(s)\big)}=\infty,
\end{equation}
then $(P_-)$ admits at least one positive solution.

\end{teorema}

 Theorem \ref{teo22} gives a complete picture of the solvability of  $(P_-)$ in the case of nonlinearities with power growth at infinity. Specifically, if  $f(s)$ behaves like $s^p$ and $g(s)$ behaves like $s^q$ as $s\to \infty$, for some $p,q>0$, then the equation $(P_-)$ has a positive solution in the whole space if and only if
  \begin{equation}\label{condi2}
p \le \max\{ 1,q \} .
 \end{equation}

 However, contrary to Theorem \ref{teo32}, in the framework of Theorem \ref{teo22} we can find examples of standard functions for  which none of the conditions in Theorem \ref{teo22} is satisfied, for instance $f(s) = s(\log s)^\alpha$, $g(s)  = s(\log s)^\beta$, with $\beta>1$ and $\alpha\in (\beta+1,\beta+2)$. We do not know if $(P_-)$ has a positive solution in these cases.
 \medskip

Let us now give some more context for Theorems \ref{teo32} and \ref{teo22}. It is an often observed feature in elliptic PDE that positive solutions in bounded domains exist in the cases when entire solutions do not exist, and vice versa. This property turns out to be verified for $(P_-)$, in the sense that Theorem \ref{teo22} contains a dual statement to the one in Theorem 1 of \cite{AGQ}. In fact, the proof of our Theorem \ref{teo22} uses a number of ideas to be found in \cite{AGQ} (and even earlier in \cite{FMQ}), together with some improvements which permit to us to remove the extra assumptions (a) and (b) in \cite{AGQ}.

On the other hand, Theorem \ref{teo32} contains a completely new statement. In particular, \re{maincond} has not appeared before, and it represents a significant improvement over the previously available results, such as, for instance, Theorem 2 in \cite{AGQ}. The existence statement in Theorem \ref{teo32} is proved through delicate analysis of the asymptotics of the ODE associated to $(P_-)$, while the nonexistence proof makes use of the definition and properties of viscosity solutions of PDE.
\medskip

Finally, an important advantage of our methods is that they readily give existence  and non-existence results for  general equations in the form \re{princ}.  Let $\mathcal{M}^+$ denote the Pucci extremal operator, and consider the equation
\begin{equation}\label{geneq1}
\begin{array}{ll}
\mathcal{M}^+(D^2 u) = H(u, |\nabla u|) & \mbox{in }\R^N,\\
\end{array}
\end{equation}
where $H=H(u,p)$ is  continuous on $[0,\infty)^2$, $H(u,0)>0$ for $u>0$, and $H$ satisfies at least one of the following inequalities
\begin{equation}\label{gencond}
f_1(u) \pm g_1(p)\le H(u,p)\le f_2(u) \pm g_2(p),
\end{equation}
for some $f_i,g_i,i=1,2$, which satisfy \re{condfg}, and all $u,p\in \mathbb{R}^+$.

Then we have the following generalizations of Theorem~\ref{teo32} and Theorem~\ref{teo22}.

\begin{teorema}\label{genteo1} Under the assumptions of Theorem~\ref{teo32} (i) (resp. Theorem~\ref{teo22} (i)), the inequality
$$
\mathcal{M}^+(D^2 u) \ge f(u) + g(|\nabla u|)\qquad (\mbox{resp. }\; \mathcal{M}^+(D^2 u) \ge f(u) - g(|\nabla u|))
$$
does not have nontrivial nonnegative solutions in the whole space.

\end{teorema}

\begin{teorema}\label{genteo2} The equation \re{geneq1} has a positive solution provided either
$$
H(u,p)\le f(u) + g(p)
$$
and $f,g$ satisfy the hypotheses of Theorem~\ref{teo32} (ii), or
$$
H(u,p)\le f(u) - g(p)
$$
and $f,g$ satisfy the hypotheses of Theorem~\ref{teo22} (ii).

\end{teorema}

We recall that the Pucci's extremal operator $\mathcal{M}^+$ has  the property that
$\mathcal{M}^+(M)=\mathop{\sup}_{A\in \mathcal{S}}\mbox{tr}(AM), $ where $
\mathcal{S}$ denotes the set of symmetric matrices whose
eigenvalues lie in the interval $[ \lambda, \Lambda ]$, for some positive constants $0<\lambda\le \Lambda$ (we can assume that $\Lambda =1$, which amounts to replacing $H$ by $H/\Lambda$). Observe that $\mathcal{M}^+(D^2 u)=\Delta u$ if $\lambda=1$.
Hence we can infer from Theorem \ref{genteo1}  that \re{intro-cond2} is sufficient for the non-existence of positive entire solutions of {\it any} semi-linear inequality
$$
\sum_{i,j=1}^N a_{ij}(x)\partial_{ij}u \ge f(u) - g(|\nabla u|).
$$
where $(a_{ij}(x))$ is a matrix with eigenvalues in $[\lambda,1]$, and $f,g$ satisfy \re{condfg}.

\section{Preliminaries}
\setcounter{section}{2}
\setcounter{equation}{0}

In the proof of our main results we will use the following comparison principle.

\begin{propo}\label{Lemmacomparison}
Assume that $f $ and $g$ verify the condition \re{condfg} and $\Omega$ is a bounded domain.
Let $u,v$ be solutions of the inequalities
\begin{equation*}
\mathcal{F}(D^2 u) - f({u})\pm g(|\nabla u|)\geq 0\quad \mbox{in } \Omega,
\end{equation*}
\begin{equation*}
 \mathcal{F}(D^2 v) - f({v})\pm g(|\nabla v|)\leq 0\quad \mbox{in } \Omega,
\end{equation*}
where $\mathcal{F}$ is an uniformly elliptic second-order operator.
If
\begin{equation}\label{limsup}
\displaystyle\limsup_{\delta(x)\rightarrow 0}\frac{u(x)}{v(x)} < 1,
\end{equation}
where $\delta(x)$ is the distance to the boundary of $\Omega$, then
\begin{equation}\label{desuv}
u\leq v\quad \mbox{in }\Omega. \end{equation}
\end{propo}

\noindent{\bf Proof}. This is standard, since the operator in the left hand side of the inequalities satisfied by $u$ and $v$ is {\it proper}. If both $u$ and $v$ are classical solutions, just evaluate the equations at a positive maximum of $u-v$. If one of $u,v$ is classical and the other weak, use a standard test function argument.
If both $u$ and $v$ are weak solutions we can apply for instance the results from Ishii-Lions \cite{IL} (in the viscosity solutions case) or from Pucci-Serrin \cite{PS} (if weak-Sobolev solutions are considered).\hfill $\Box$
\medskip

Next we give some preliminary properties
of solutions to the Cauchy problem
\begin{equation}\label{cauchy}
\left\{
\begin{array}{l}
 u''+\ds  \frac{N-1}ru'=f(u)\pm g(u'),\\[0.5pc]
u(0)=u_0>0,\quad u'(0)=0,
\end{array}
\right.
\end{equation}
where  $f$ and $g$ are functions satisfying \re{condfg}. Note that if $u(r)$ is a solution to \re{cauchy} and $u^\prime\ge 0$, then $u(|x|)$ is a solution to $(P_\pm)$.
Recall that, thanks to the continuity of $f$ and $g$, Peano's theorem guarantees that \re{cauchy} has at least one solution defined in a right neighborhood of zero.

The following result is essentially known, but we give a full proof, for readers' convenience.

\begin{propo}\label{teo6}
Let $u$ be a  solution to \eqref{cauchy} defined on some interval $(0,R)$. Then
$u(r)>0$, $u'(r)>0$, $u''(r)\ge 0$ for $r\in(0,R)$.
In particular, $u$ and $u'$ are increasing functions and
\begin{equation}\label{cotaR}
u(r) \leq u_0+Ru'(r), \qquad \mbox{for each }\;r\in(0,R).
\end{equation}
\end{propo}

\begin{proof}
Letting $r\to0$ in the equation we obtain $u''(0)=\frac{f(u_0)}{ N}>0$,
so that $u''(r)>0$ if $r > 0$ is sufficiently small. This implies $u'(r)>0$
for $r>0$ close enough to zero. Assume there exists $r_1>0$ with $u'(r)>0$ if $r\in(0,r_1)$
and $u'(r_1)=0$. Then we would have $u''(r_1)\leq0$, while  the equation gives
$$
u''(r_1)={f\big(u(r_1)\big)}>0,
$$
a contradiction. Hence $u'(r)>0$ for all $r\in (0,R)$.

\medskip

Let us now deal with the sign of $u''$. We begin with \re{cauchy} with
minus sign in it:
$$
 u''+  \frac{N-1}{r}u'=f(u)-g(u').
$$
Assume $u''(\bar r)<0$ for some $\bar  r>0$. Let
$$
 r_0=\inf\{\tilde{r}:\ u''(r)<0 \hbox{ in } (\tilde{r},\bar r)\}.
$$
Since $u''(0) > 0$, we have $r_0>0$ and $u''( r_0)=0$. Moreover, $u'$ is decreasing in $r\in [r_0, r_0+\epsilon]$, for some $\epsilon>0$. Since we already know $u$ is increasing, we have by \re{condfg} that
$$
u''={f(u)}-\frac{N-1}{r}u'-{g(u')}
$$
is increasing. But then $u''( r_0)=0$ implies that $u''\ge0$ if $r$ is larger than and close to $r_0$, a contradiction. Thus $u''\geq0$ in $(0,R)$.

Next, we turn to \re{cauchy} with the plus sign:
$$
 u''+ \frac{N-1}{r}u'= f(u)+g(u'),
$$
in which case we can even show that $u''$ is strictly positive. First, if $N=1$ then $u''$ is obviously positive, since $f$, $g$, $u$ and $u'$ are. So we can assume $N>1$. Similarly to the above we assume for contradiction that there exist $\varepsilon, r_0>0$ such that $u''>0$ in $(r_0-\varepsilon, r_0)$, and $u''(r_0)=0$. Then the equation
$$
 u''+ (N-1){r}^{-1}u'= h(r)\quad \mbox{ in }\;[r_0-\varepsilon, r_0]
 $$
 for some increasing function $h$ clearly implies that
\begeqaet
 0&\ge& \limsup_{h\searrow0} \frac{u''(r_0-h) -  u''(r_0)}{-h}\\
 &\ge& \liminf_{h\searrow0} \frac{h(r_0-h) -  h(r_0)}{-h} -(N-1)\left( \frac{u'}{r}\right)^\prime_{r=r_0} \ge (N-1) \frac{u'(r_0)}{r_0^2} >0,
\eneqaet
 a contradiction.

Finally, since $u''\ge 0$, we have that $u'$ is nondecreasing. Thus, for $r\in(0,R)$
$$
u(r)=u_0+\int_0^ru'(s)\,ds\leq u_0+Ru'(r),
$$
which concludes the proof.
\end{proof}


\section{Proof of the main theorems}
\setcounter{section}{3}
\setcounter{equation}{0}

We begin with the existence statement in Theorem \ref{teo32}.

\noindent {\bf Proof of Theorem \ref{teo32}, Part (ii).}
Let $u$ be a solution to
\begin{equation}\label{cauchy3}
\left\{
\begin{array}{l}
 u''+\ds  \frac{N-1}ru'=f(u)+g(u'),\\[0.5pc]
u(0)=u_0>0,\ u'(0)=0,
\end{array}
\right.
\end{equation}
defined on some maximal interval $(0,R)$.
We claim that $R=\infty$, which implies that $u(|x|)$ is an entire solution of $(P_+)$. Assume for contradiction that $R$ is finite. Then Proposition \ref{teo6} implies that $u'(r)\to \infty$ as $r\to R$.

Suppose first  that the increasing function $u(r)$ is bounded by some constant $C_1$ as $r\to R$.
 Then  by Proposition \ref{teo6} we get $ u''\leq g(u')+f(C_1)$. If  we divide this inequality by $g(u')$  then integrate between $r_0$ and $r$ for some
arbitrary $0<r_0<r<R$, we obtain after the change of variables $s=u^\prime(r)$
\begin{equation}\label{contra1}
 \int_{u'(r_0)}^{u'(r)} \frac{ds}{g(s)}  \leq C_0 (r-r_0)\qquad (\mbox{with }C_0:=1+f(C_1)/g(u^\prime(r_0)).
\end{equation}
Then letting $r\to R$ we obtain a contradiction with \re{gKO}.

Therefore we can assume  that $u'(r)\to \infty$ and  $u(r)\to \infty$ as $r\to R$.
Now we  define
$$A(r)=\frac{r^{N-1}u'}{\sqrt{F(u)}},$$
and we split the proof in three different cases, according to the asymptotic behavior of $A(r)$ as $r\to R$.
\medskip

\noindent {\it Case 1}. Assume $A(r)$ is bounded as $r\to R$. Then by integrating between $R/2$ and any $r\in (R/2, R)$ we get
$$\left(\frac{R}{2}\right)^{N-1}\int_{u(R/2)}^{u(r)}\frac{ds}{\sqrt {F(s)}}= \left(\frac{R}{2}\right)^{N-1}\int_{R/2}^{r}\frac{u'(s)ds}{\sqrt {F(u(s))}} \leq \int_{R/2}^{r} A(s)\,ds\leq CR,$$
for some positive constant $C$. Letting $r\to R$ leads us to a contradiction with \re{KO}.
\medskip

\noindent {\it Case 2}. Assume  $A(r) \to \infty$ as $r\to R$.
Define $v(r)=u'(r)$ and $H(r)=\sqrt{F(u(r))}$ and notice that \eqref{cauchy3} can be recast as
\begin{equation}\label{cauchy4}
\left\{
\begin{array}{l}
\ds v'-g(v)+\frac{N-1}{r} v = 2 \frac{H}{v}H',\\[0.5pc]
v(0)=0.
\end{array}
\right.
\end{equation}
Note that the assumption of Case 2 is equivalent to $\ds\frac{H}{v}\to0 $ as $r\to R$.

Let $r_0>0$ be such that $v\ge 1$ in $(r_0,R)$.
We now define $$w=S^{-1}(v),$$
where $S=S(t)$ is the solution of
\begin{equation}\label{cauchy44}
\left\{
\begin{array}{l}
S'(t)=g(S(t)),\\[0.5pc]
S(1)=1.
\end{array}
\right.
\end{equation}
 Observe that $S$ is bijective from $[1,\infty)$ to  $[1,\infty)$, thanks to the assumption \eqref{gKO}. From \eqref{cauchy4} we get
$$w'\leq 1  +2 \frac{H H'}{vS'(w)},$$
that is,
\begin{equation}\label{evS}
[S^{-1}(v)]'\leq 1+2 \frac{H}{v}\frac{g(H)}{g(v)}[S^{-1}(H)]'
\end{equation}

Since  $\frac{H}{v} \to 0$ as $r\to R$ and $g$ is increasing we find that there exists $r_0<R_0<R$ such that
$$2\frac{H g(H)}{v g(v)}\leq \frac {1}{2}\quad \mbox{in}\quad [R_0,R).$$

Now we integrate \eqref{evS} between $r_0$ and any $r\in (R_0,R)$, to get
\begeqa
S^{-1}(v)&\leq& 1+{R}+\int_0^{R_0} 2\frac{H g(H)}{v g(v)}[S^{-1}(H)]'ds+\frac{1}{2}S^{-1}(H(r))-\frac{1}{2}S^{-1}(H(R_0))\nonumber\\
&=&C(R_0,R)+\frac{1}{2}S^{-1}(H(r))\label{fret},
\eneqa
where the quantity $C(R_0,R)$ is independent of $r\in (R_0,R)$.

Since $F$ is the primitive of the increasing function $f$ we obviously have
$H(r)\to \infty $ as $r\to R$. Hence by the definition of $S$ and \re{gKO}
we have $$S^{-1}(H)\to \infty\quad\mbox{ as }\; r\to R.$$ It then follows from \re{fret} that we can find $R_1\in(R_0, R)$ such that
$$S^{-1}(v)<S^{-1}(H)\quad\mbox{ in }\;(R_1,R).$$ But $S^{-1}$ is increasing, so
 $v<H$ in $(R_1,R)$,  a contradiction with $\ds\frac{H}{v} \to 0$ as $r\to R$.
\medskip

\noindent {\it Case 3}. Assume we are in neither of the previous two cases. This means $A(r)$ oscillates in the sense that
\begin{equation}\label{limsup2}
 \limsup_{r\to R} A(r)=+\infty \quad\mbox{ and }\quad \liminf_{r\to R} A(r)<+\infty.
\end{equation}
Thus there exists $A_0$ such that for each $\hat A\geq A_0$ we can find two sequences
$s_n, r_n\to R$ such that
\begin{equation}\label{ll}
A(s_n)=A(t_n)=\hat A,\qquad A'(s_n)\geq0,\qquad\mbox{and}\qquad A'(t_n)\leq 0.
\end{equation}

After a computation we find
\begin{equation*}
A'(r)=\frac{r^{N-1}f(u(r))}{\sqrt{F(u(r))}}W(r)
 \end{equation*}
 where we have used ${(r^{N-1}u')}^\prime = r^{N-1}(f(u) + g(u'))$ and we have set
 \begin{equation*}
W(r) := 1+ \frac{g(u')}{ f(u)}-\frac{(u')^2}{ 2 F(u)} = 1+ (\bar A(r))^2\left(\frac{g\left(\bar A(r)\sqrt{F(u(r))}\right)}{(\bar A(r))^2 f(u(r))}-\frac{1}{ 2}\right),
\end{equation*}
and
$$
\bar A(r)=\frac{A(r)}{r^{N-1}}.
$$
Note that $\bar A(s_n),\bar A(t_n) \to \tilde A:=\hat A R^{1-N}$. Therefore for each $\varepsilon>0$ there exists $n_0\in \mathbb{N}$ such that for all $n\ge n_0$ we have
$$
 W(t_n)\ge 1 + \bar A(t_n)\left(\delta^2\frac{g\left((\tilde A-\varepsilon) \sqrt{F(u)}\right)}{(\tilde A-\varepsilon)^2 f(u)}-\frac{1}{ 2}\right),
 $$
 $$
 W(s_n)\le 1 + \bar A(s_n)\left(\delta^{-2}\frac{g\left( (\tilde A+\varepsilon) \sqrt{F(u)}\right)}{(\tilde A+\varepsilon)^2 f(u)}-\frac{1}{ 2}\right),
$$
where $\delta : = (\tilde A-\varepsilon)/(\tilde A+\varepsilon)\to 1$ as $\epsilon\to0$. It is now obvious that we can fix $\varepsilon>0$ small enough and, if necessary, $\hat A$ so large that the condition \re{maincond} implies that either $W(s_n)<0$ or $W(t_n)>0$ for  large $n$, which is a contradiction with \re{ll}.

We have reached a contradiction in all three cases,
therefore  $u$ is defined for all $r>0$. Part (ii) of Theorem \ref{teo32} is proved.\hfill $\Box$
\medskip

{\noindent {\bf Remark 2}.} Let us now show that any of the two conditions given in Remark 1 can replace \re{maincond}.

First, it is clear that the same argument which lead us to the contradiction \re{contra1} can be used in case
$$
\limsup_{s\to\infty} \frac{f(as)}{g(s)}<\infty\quad \forall\:a>0.
$$
 Indeed, then by Proposition \ref{teo6} and $u^\prime(r)\to\infty$ as $r\to R$,
$$
f(u(r))\le f(u_0 + R u^\prime(r))\le f(2Ru^\prime(r)) \le C g(u^\prime(r))
$$
for all $r$ in some left neighborhood of $R$.

Second, if $g$ has at most linear growth at infinity, we observe that \re{cauchy4} implies
$$
\tilde v^\prime -C\tilde v\le \tilde v^\prime -g(\sqrt{\tilde v})\sqrt{\tilde v}\le \tilde H^\prime (r),
$$
where $\tilde v = v^2$ and $\tilde H = H^2$, and so
$$
\tilde v(r) \le e^{Cr}\int_{r_0}^{r} e^{-Cs}\tilde H^\prime (s)\,ds\le R e^{CR} \tilde H(r),
$$
which means we are in Case 1 above.
\medskip

\noindent {\bf Proof of Theorem \ref{teo32}, Part (i)}. We will use the following lemma. For the reader's convenience we recall that if $u(r)$ is a solution of \re{cauchy} then $u(|x|)$ is a solution of $(P_+)$.

\begin{lema}\label{t1} Under the assumption of Theorem \ref{teo32}, Part (i), any solution of \re{cauchy} exists on a maximal interval $(0,R)$ where $R=R(u_0)<\infty$ is such that $u^\prime(r)\to \infty$ as $r\to R$. In addition
\beeq\label{condexp}
u(r)\to \infty\quad \Longleftrightarrow \quad \int^\infty \frac{s}{g(s)}ds=\infty.
\eeq
\end{lema}

\begin{proof} By integrating the equation ${(r^{N-1}u')}^\prime = r^{N-1}(f(u) + g(u'))$ we get
\begeqaet
u^\prime(r) &=& \frac{1}{r^{N-1}} \int_0^r s^{N-1} ( f(u(s)) + g(u^\prime(s)) ) ds\\
&\le & \frac{1}{r^{N-1}} ( f(u(r)) + g(u^\prime(r))\int_0^r s^{N-1} ds\\
&=& \frac{r}{N} ( f(u(r)) + g(u^\prime(r)),
\eneqaet
where we have also used \re{condfg} and Proposition \ref{teo6}. Inserting this inequality into \re{cauchy} we then obtain
\begin{equation}\label{cosa2}
u''\geq\frac{1}{ N}\big(f(u)+g(u')\big).
\end{equation}
Multiplying  $u''\geq N^{-1} f(u) $ by $u^\prime$ we get
$$
({u'}^2)^\prime \ge 2N^{-1}(F(u))^\prime$$ and we easily deduce
$$\int_{u_0}^{u(r)}(F(s)-F(u_0))^{-1/2}ds\geq \sqrt{2N^{-1}}r. $$
Similarly we can divide $u''\geq N^{-1}g(u')$
 by $g(u')$ and integrate between some $r_0>0$ (so that  $u'(r_0)>0$) and an arbitrary $r>r_0$, to get
\begin{equation}\label{ggg}
 N \int_{u'(r_0)}^{u'(r)} \frac{ds}{g(s)}  \geq r-r_0.
\end{equation}
Under the assumptions of part I of Theorem \ref{teo32} one of the integrals in the left-hand sides of the last two inequalities is bounded above, so $r$ is also bounded above, that is, $R=R(u_0)$ is finite.

Recall we proved in Proposition \ref{teo6} that $u, u^\prime$ are increasing and $u(r)\le C(1+u^\prime(r))$. Hence $u^\prime(r)\to \infty$ as $r\to R$.

Next, if the integral in the right-hand side of \re{condexp} is finite, the inequality
$$
\frac{u'' u'}{g(u')}\ge N^{-1} u'
$$
implies, after integration,
$$
u(r)-u(r_0)\le C\int_{u'(r_0)}^{u^\prime(r)} \frac{s}{g(s)} ds <\infty
$$
for any $0<r_0<r<R$.

On the other hand, if $u(r)\le C$ as $r\to R$  then
$$
u'' \le u'' + \frac{N-1}{r} u' \le g(u') + f(u) \le  g(u') + f(C)
$$
in $(0,R)$. We multiply this inequality by $u'$ and integrate, to obtain
$$
\int_{u'(r_0)}^{u^\prime(r)} \frac{s}{g(s)} ds \le u(r) - u(r_0),
$$
for any $0<r_0<r<R$.  Letting $r\to R$ in this inequality we see that the integral in the right-hand side of \re{condexp} is finite, since we already know that $u^\prime(r)\to \infty$ and $u(r)\le C$ as $r\to R$.
\end{proof}

We will also make use of the following simple analysis lemma.

\begin{lema}\label{analysis}
Assume $h:[1,\infty) \to (0,\infty)$ is a nonincreasing function such that
$$
\int_1^\infty h(s)ds <\infty.
$$
Then $th(t)\to 0$ as $t\to \infty$.
\end{lema}

\begin{proof}
It is clear that $h(t)\searrow 0$ as $t\to \infty$. Assume for contradiction that there exists a sequence $t_n\to \infty$ such that
$$
t_nh(t_n)\ge \epsilon_0>0
$$
as $n\to\infty$. Without restricting the generality we can assume that $t_{n+1}\ge 2t_n$ for each $n\ge 1$. Then we have
\begeqaet
\int_1^\infty h(s)ds &\ge& \sum_{n=1}^\infty \int_{t_n}^{t_{n+1}} h(s)ds  \ge \sum_{n=1}^\infty h(t_{n+1})(t_{n+1}-t_n)\\
&\ge &  \sum_{n=1}^\infty \epsilon_0 \left( 1 - \frac{t_n}{t_{n+1}}\right) \ge \sum_{n=1}^\infty \epsilon_0/2 =\infty,
\eneqaet
a contradiction.
\end{proof}

\bigskip

We will end the proof with the help of the following proposition, which contains a stronger statement than Part (i) of Theorem \ref{teo32}.

\begin{propo}\label{pro}
Under the assumption of Theorem \ref{teo32}, Part (i), the inequality
 \begin{equation}\label{EQ}
{\mathcal M}^+(D^2 w)\ge g(|\nabla w|)+f(w)\qquad \mbox{in}\quad \R^N
\end{equation}
does not have a non-negative, non-trivial viscosity subsolution $w$.
\end{propo}

\noindent
{\bf Proof.} Suppose the statement is false.
We may assume that $w(0)>0$. Let $u$ be a solution of \re{cauchy} with $u_0=w(0)/2$ (so that Proposition \ref{teo6} and Lemma \ref{t1} apply).
Set $v(x)=u(|x|)$ and observe that, by Proposition \ref{teo6}, $v$ is a convex function. Then, since the Pucci maximal operator has the property that \begin{equation}\label{proppucci}
{\mathcal M}^+(M) = \lambda \sum_{e_i<0} e_i + \sum_{e_i>0} e_i,
\end{equation}
where $e_i$ denote the eigenvalues of the symmetric matrix $M$, we have
\begin{equation}\label{EQq}
{\mathcal M}^+(D^2 v)= g(|\nabla v|)+f(v)\qquad \mbox{in}\quad B(0,R)
\end{equation}
(here $R$ is the number from Lemma \ref{t1}).

First we observe that $u$ is bounded as $r\to R$ -- indeed, if $u(r)\to \infty$ as $r\to R$ we apply the comparison principle in $B(0,R-\varepsilon)$, where  $\varepsilon>0$ is chosen sufficiently small so that $v>w$ on $\partial B(0,R-\varepsilon)$, and get a contradiction with $u_0=w(0)/2$.

By applying in the same way the comparison principle in $B(0,R)$ we see that there exists $\bar x\in \partial B(0,R)$ such that
$$
w(\bar x)>v(\bar x).
$$

Then, there is
$a>0$ such that the function $v_a=v+a$ satisfies
\begin{equation}
\label{b1}
w(x)\le v_a(x) \qquad \mbox{for all }\quad x\in \partial B(0,R),
\end{equation}
and there is $x_0\in\partial B(0,R) $ such that
\begin{equation}
\label{b}
w(x_0)=v_a(x_0).
\end{equation}
We observe that, since $f$ is increasing, the function $v_a$ is a supersolution of (\ref{EQ}), that is,
\begin{equation}\label{EQv}
{\mathcal M}^+(D^2v_a)\le g(|\nabla v_a|)+f(v_a)\qquad \mbox{in}\quad B(0,R).
\end{equation}
Then, we use (\ref{b1}) and the comparison principle, to obtain that
 \begin{equation}
\label{b3}
w(x)\le v_a(x) \qquad \mbox{for all }\quad x\in B(0,R).
\end{equation}

Next, for any  $m>0$, we consider the radially symmetric conical function
$$
\varphi_m(x)=m(|x|-R)+a+u(R)=m(|x|-R)+v_a(x_0).
$$
{\bf Claim:} For all $m>0$ and $r_0>R$
there exists $\bar x\in B(0,r_0)\setminus B(0,R)$ such that
$$
\varphi_m(\bar x)<w(\bar x).
$$
In order to prove the claim, we assume there is $r_0>R$ and $m>0$ such that
$$
w(x)\le \varphi_m(x), \qquad \mbox{for all }\quad x\in B(0,r_0)\setminus B(0,R).
$$
Then, for any $k>0$, the function
$$
\varphi(x)=\varphi_{2m}(x)-k(|x|-R)^2
$$
is a test function at $x=x_0$ for the inequality \re{EQ} satisfied by $w$.
Indeed, for $x$ such that $|x|\in (R, R+m/k)$ we clearly have
$$
 \varphi(x)\ge \varphi_m(x)\ge w(x),
 $$
 while \re{b3} and $u^\prime(r)\to \infty$ as $r\to R$ imply that
$$
 \varphi(x)\ge v_a(x)\ge w(x)
 $$
 for all $x$ such that $|x|\in (R-\sigma,R)$, for some $\sigma>0$. Moreover
 $$
\varphi(x_0)=\varphi_{2m}(x_0)-k(|x_0|-R)^2=a+u(R)=w(x_0).
$$
Thus, we may test the equation (\ref{EQ}) at $x_0$ with $\varphi$, so we have
$$
{\mathcal M}^+(D^2\varphi(x_0))\ge g(|\nabla \varphi(x_0)|)+f(\varphi(x_0)).
$$
However, for any large $k$ we have
$$
 {\mathcal M}^+(D^2\varphi(x_0))=
\lambda(-2k+2(N-1)\frac{m}{R})<0,
$$
 a contradiction which proves the claim.

Now we continue the proof of the lemma considering the function $\varphi_m$.
Given $m>0$ and $x$ such that $|x|\ge R$, we have that
$$
{\mathcal M}^+(D^2\varphi_m(x))=\Lambda (N-1)\frac{m}{|x|}\le \Lambda (N-1)\frac{m}{R}.
$$
On the other hand, we clearly have
$$
g(|\nabla\varphi_m(x)|)+f(\varphi_m(x))\ge g(m).
$$
Observe now that if \eqref{gKO} does not hold then by Lemma \ref{analysis}
\begin{equation}\label{G}
\lim_{m\to\infty}\frac{g(m)}{m}=\infty.
\end{equation}
 Thus, using (\ref{G}), for all $m$ large enough we have
\begin{equation}\label{contra}
{\mathcal M}^+(D^2\varphi_m (x))<g(|\nabla\varphi_m(x)|)+f(\varphi_m(x))\qquad\mbox{for all }\quad x\in B(0,r_0)\setminus B(0,R).
\end{equation}
We may in addition assume that $m$ and $r_0>R$ are fixed so that $\varphi_{m}\ge w$ on $\partial B(0,r_0)$. However, by what we proved we know that the graph of $w$ is above the graph of $\varphi_m$ somewhere in $B(0,r_0)\setminus B(0,R)$. Then for each small $\varepsilon>0$ the function
$$
\tilde \varphi_{m}(x)=\varepsilon+\varphi_{m}(x)
$$
satisfies
\begin{eqnarray}
w(x)&<& \tilde \varphi_{m}(x)\qquad\mbox{for all }\quad x\in \partial B(0,R) \label{pf1},\\
w(x)&<& \tilde \varphi_{m}(x)\qquad\mbox{for all }\quad x\in \partial B(0,r_0)\label{pf2}\mbox{ and}\\
w(\bar x)&>& \tilde \varphi_{m}(\bar x)\qquad\mbox{for some } \bar x \in B(0,r_0)\setminus B(0,R).\label{pf3}
\end{eqnarray}
If we start increasing $\varepsilon$, conditions (\ref{pf1}) and (\ref{pf2})
remain true, while for certain value $\varepsilon=\varepsilon_0$ the condition (\ref{pf3}) changes into
$$
w( x)\le \tilde \varphi_{m}( x)\qquad\mbox{for all } \bar x \in B(0,r_0)\setminus B(0,R),
$$
and there exists $\bar x \in B(0,r_0)\setminus \overline{B(0,R)}$ such that
$$
w(\bar x)= \tilde \varphi_{m}(\bar x).
$$
Thus, the function $\tilde \varphi_{m}$ is a test function at $\bar x$ for the equation satisfied by $w$, and we get a
contradiction with (\ref{contra}), thus completing the proof of the proposition.

Theorem \ref{teo32} is proved. \hfill $\Box$
\bigskip

\noindent {\bf Proof of Theorem \ref{teo22}, Part (i)}.  We will use the following lemma.

\begin{lema}\label{kwak}
There exists a strictly increasing  supersolution ${\bar u}$ to
\begin{equation}\label{cauchy2}
\left\{
\begin{array}{l}
 u''+\ds  \frac{N-1}ru'\le f(u)- g(u'),\\[0.5pc]
u(0)=\bar u_0,\ u'(0)=0,
\end{array}
\right.
\end{equation}
which ceases to exist at a finite $R$, with ${\bar u}(r)\to \infty$ as $r\to R$,  provided  $\bar u_0$ is taken large enough.
\end{lema}

\begin{proof}
We will assume for the moment that $R\le 1/2$, and will search for a supersolution in
the form
\begin{equation*}
\bar{u}(r)=\phi(R^2-r^2),
\end{equation*}
where $\phi(t)\to\infty$ as $t\to 0$, $t>0$. It is not hard to show that the function $\bar{u}$ thus obtained will be a supersolution of \eqref{cauchy2}
if
\begin{equation}\label{fwer}
 4r^2\phi''-2N \phi' +g\big(2r|\phi'|\big)
\leq f\big(\phi \big),
\end{equation}
where $'$ stands for differentiation with respect to $t=R^2-r^2$.
Assume in addition that $\phi'<0$, $\phi''\ge0$. Then it suffices to have
\begin{equation}\label{supersolucion}
\phi''+2 N|\phi'|+g(|\phi'|)\leq f(\phi).
\end{equation}
Let now $\phi(t)$ be the function defined by the implicit relation
$$
\int_{\phi(t)}^\infty \frac{ds}{\Gamma^{-1}(F(s))} = t.
 $$
 Thanks to our hypotheses on $f$ and $g$ as well as \eqref{intro-cond2}, $\phi(t)$ is well defined and we have
  $\phi'<0$,  $\phi(t)\to\infty$ as $t\to 0$, and
\begin{equation}\label{ecuacion-phi}
\ds\Gamma(|\phi'|)=F(\phi).
\end{equation}
Notice that by Lemma \ref{analysis} the convergence of the integral in \eqref{intro-cond2} implies
\begin{equation*}
\frac{\Gamma^{-1}(F(s))}{s}\rightarrow\infty\qquad \mbox{as }\; s\to \infty,
\end{equation*}
and hence
\begin{equation*}
\frac{|\phi'(t)|}{\phi(t)}\rightarrow\infty\qquad \mbox{as }\;t\to 0.
\end{equation*}
In particular, there exists $\e >0$ such that $\phi(t) \leq \frac12|\phi'(t)|$ if $t\in (0,\e)$.
Restrict $R$ further to have $R^2 \leq \e$, so that $t=R^2-r^2\in (0,\e)$  for $r\in (0,R)$.

Let us check
that $\phi$ verifies \re{supersolucion}. By differentiating with respect to $t$ in \eqref{ecuacion-phi}
we get
\begin{equation*}
\phi''=\,f(\phi)\,\frac{|\phi'|}{2 g(2|\phi'|)+4N|\phi'|}.
\end{equation*}
We deduce then that $\phi''>0$ and
\begin{equation}\label{sup1}
\phi''\leq\frac{1}{4N}\,f(\phi).
\end{equation}
On the other hand,  by the monotonicity of $f$ and $g$ we have
\begin{equation*}
F(t)=\ds\int_0^tf(s)\,ds\leq f(t)t,
\end{equation*}
and
\begin{equation*}
\Gamma(t)\geq\int_t^{2t}g(s) ds+ 2N t^2\geq t g(t)+2N t^2.
\end{equation*}
Then
\begin{equation}\label{sup2}
 g(|\phi'|)+2 N |\phi'| \le \frac{\Gamma(|\phi'|)}{|\phi'|} = \frac{F(\phi)}{|\phi'|} \leq
  f(\phi) \frac{ \phi}{|\phi'|} \leq  \frac12 f(\phi).
\end{equation}
By adding \eqref{sup1} and \eqref{sup2} we obtain \eqref{fwer}, and the lemma is proved.
\end{proof}
\medskip

We now prove part (i) of Theorem \ref{teo22}.
Assume now that $u$ is a nontrivial  solution to  $(P_-)$ defined on the whole space. Let $\tilde u$ be a solution to  \eqref{cauchy2} with $\tilde u_0=\ds\frac{u(0)}{2}$.  Note that Proposition \ref{teo6} applies to $\tilde u$. In addition, $\tilde u'>0$, $\tilde u''\ge 0$ together with the equality satisfied by $\tilde u$ imply
$$
\tilde u^\prime(r)\le r(N-1)^{-1} f(\tilde u(r)).
$$
Hence if $\tilde u$ exists on some maximal interval $(0, \bar R)$ with $\bar R<\infty$, we have $\tilde u(r)\to \infty$ as $r\to R$. This implies that   $$u(x)< \tilde u(|x|)\quad\mbox{ on }\;\partial B(0,r_0),$$
 for some $r_0$ smaller than  and close to $\bar R$. Then the comparison principle applies in $B(0,r_0)$ and leads to a contradiction with $\tilde u(0)<u(0)$.

Therefore $\tilde u$ is globally defined and $\tilde u(r)\to \infty$ as $r\to \infty$, since $\tilde u$ is increasing and convex. We already know by the comparison principle that $$u(x)\not\le \tilde u(|x|)\quad\mbox{ on }\quad\partial B(0,r),$$ for all $r\in (0, \infty)$. This implies that there exists a sequence $x_n\in \mathbb{R}^N$ with $|x_n|\to \infty$, such that $u(x_n)\to \infty$. Fix $n_0$ so large that $u(x_{n_0})> \bar u_0$, where $\bar u_0$ is the number we obtained in Lemma \ref{kwak}. We then can repeat the argument from the previous paragraph, replacing $\tilde u(|x|)$ by $\bar u(|x-x_{n_0}|)$ ($\bar u$ is the function from Lemma \ref{kwak}), and reach a contradiction again.

\bigskip

\noindent {\bf Proof of Theorem \ref{teo22}, Part (ii)}. This proof is quite simple, we include it for completeness. Let $u$ be a solution of the problem \begin{equation}\label{cauchy22}
\left\{
\begin{array}{l}
 u''+\ds  \frac{N-1}ru'=f(u)- g(u'),\\[0.5pc]
u(0)= u_0>0,\ u'(0)=0.
\end{array}
\right.
\end{equation}
By Proposition \ref{teo6} we  get $ f(u)-g(u')\ge 0$, that is,
\begin{equation}\label{22}
u'\leq g^{-1}(f(u)).
\end{equation}
Using again Proposition \ref{teo6} and \re{condfg} we see that if $u$ is defined on a maximal interval $(0,R)$ with $R<\infty$ then  $u(r)\to \infty$ as $r\to R$.

By Proposition \ref{teo6} and \re{cauchy22} we also have
$$
u^{\prime\prime}\le f(u),  $$
which implies $\left((u^\prime)^2\right)^\prime \le 2 \left(F(u)\right)^\prime$, and hence
\begin{equation}\label{23}u^\prime \le \sqrt{2 F(u)}.\end{equation}

Finally, by integrating \re{22} and \re{23} we get
$$\max\left\{ \int_{u_0}^{u(r)}\frac{ds}{\sqrt{F(s)}}, \int_{u_0}^{u(r)}\frac{ds}{g^{-1}(f(s))}\right\}\leq \sqrt{2}  r,\qquad r\in (0,R).$$
Letting $r\to R$ we get a contradiction with the assumption of the theorem. Hence  u is  defined  on $\mathbb{R}^+$ and then
 $v(x)=u(|x|)$ is a solution of $(P_-)$. \hfill $\Box$

\medskip

In the end we give the proofs of Theorems \ref{genteo1} and \ref{genteo2}.
\medskip

\noindent {\bf Proof of Theorem \ref{genteo1}}. The first part of the theorem
is exactly Proposition \ref{pro}. For the second part we observe that in Lemma \ref{kwak} we constructed an increasing and convex supersolution $u$, and hence the function  $v(x)=u(|x|)$ is a supersolution for
$$
\mathcal{M}^+(D^2v)\le f(v)-g(|\nabla v|).
$$
Therefore we can repeat  the proof of Theorem \ref{teo22} (i) without any changes.\hfill $\Box$
\medskip

\noindent {\bf Proof of Theorem \ref{genteo2}}. We consider the initial value problem
\begin{equation}\label{geninit}
\left\{
\begin{array}{l}
u^{\prime\prime} = M\left(\ds -\frac{N-1}{r}m(u^\prime) + H(u,u^\prime)\right)\\
u(0)=u_0>0,\quad u^\prime(0) = 0,
\end{array}
\right.
\end{equation}
where the functions $M(s), m(s)$ are defined as follows
$$
M(s) = \left\{ \begin{array}{rcl} s&\mbox{ if }& s\ge0\\ s/\lambda&\mbox{ if }& s\le0,\end{array} \right.
\qquad
 m(s) = \left\{ \begin{array}{rcl}    s&\mbox{ if }& s\ge0\\ \lambda s&\mbox{ if }& s\le0. \end{array} \right.
 $$

Then it is easy to see, with the help of \re{proppucci}, that if $u(r)$ is a solution of \re{geninit} defined on $\mathbb{R}^+$ then $u(|x|)$ is a solution of \re{geneq1} (see for instance Section 2 in \cite{FQ}).

Next, we observe that the usual proof of Peano's theorem applies to the singular initial value problem
problem
\begin{equation}\label{genin}
\left\{
\begin{array}{l}
(r^{N-1}{\hat u}^\prime)^\prime =r^{N-1} H(\hat u,{\hat u}^\prime)\\
{\hat u}(0)=u_0>0,\quad {\hat u}^\prime(0) = 0,
\end{array}
\right.
\end{equation}
(which is \re{geninit} with $\lambda=1$) and permits to us to construct a solution of \re{genin} in a right neighborhood of zero. More precisely, the solution is obtained as a fixed point of a continuous map from a convex compact subset of a vector space into itself.

Furthermore, observe that the hypotheses on $H$ imply $H(u_0,0)>0$.  We see that letting $r\to 0$ in \re{geninit} and assuming ${\hat u}^{\prime\prime}(0)\le0$ gives ${\hat u}^{\prime\prime}(0)\ge H(u_0,0)$, a contradiction. Hence ${\hat u}^{\prime\prime}(0)>0$, which implies that ${\hat u}$ is a solution of \re{geninit} in some interval $(0,\delta)$, $\delta>0$. Then by applying again Peano's theorem to \re{geninit} we get a solution $u$ to this equation in some (maximal) interval $(0,R)$.

As in the proof of Proposition \ref{teo6} we see that $u^{\prime}(r)>0$ for all $r\in (0,R)$, and $u$ is strictly increasing.
Therefore if we are unable to extend $u$ in a right neighborhood of $R$, this may happen either because $u(r)\to \infty$ as $r\to R$ or because $u$ is bounded as $r\to R$ but  $u^\prime(r_n)\to \infty$ for some sequence $r_n\to R$. We claim that in the second case we actually have $u^\prime(r)\to \infty$ as $r\to R$. Indeed if this claim were not true, there would exist a strictly increasing sequence $s_n\to R$ and a fixed number $A>0$ such that $u^\prime(r)\in (0, A+1]$ for $r\in [s_{2k}, s_{2k+1}]$, $k\in \mathbb{N}$, as well as $u^\prime(s_{2k})= A$ and $u^\prime(s_{2k+1})= A+1$. Then the equation \re{geninit} implies that $u^{\prime\prime}$ is bounded in $[s_{2k}, s_{2k+1}]$ independently of $k\in \mathbb{N}$, which contradicts $u^\prime(s_{2k+1})-u^\prime(s_{2k})=1$ and $s_{2k+1}-s_{2k}\to 0$ as $k\to\infty$.

We thus conclude that if $u$ is not defined on $\mathbb{R}^+$ then there exists $R>0$ such that either $u(r)\to\infty$ or $u^\prime(r)\to\infty$ as $r\to R$. Observe that in the case when we assume an inequality with a minus sign in front of the gradient term we have
$$H(u,|\nabla u|)\le  f(u) - g(|\nabla u|)\le f(u),$$
so the equation \re{geninit} and $u^\prime >0$ easily imply $u^{\prime\prime}\le f(u)$, and it is easy to check that
we necessarily have $u(r)\to\infty$ as $r\to R$, by \re{23}.

To summarize, either the function $v(x)=u(|x|)$ we just constructed is an entire positive solution of \re{geneq1}, thus proving Theorem \ref{genteo2}, or $v$ is a supersolution of the inequality
$$
\mathcal{M}^+(D^2v) \le f(v)+  g(|\nabla v|)\qquad\mbox{resp. }\; \mathcal{M}^+(D^2v) \le f(v)-  g(|\nabla v|)
$$
and either $u(r)\to\infty$ as $r\to R$ (this necessarily holds in case of a minus sign) or $u^\prime(r)\to\infty$ as $r\to R$.

On the other hand, the hypotheses of Theorem \ref{genteo2} and Theorems \ref{teo32}-\ref{teo22} imply that the problem
$$
\Delta w = f(w)+  g(|\nabla w|)\qquad\mbox{resp. }\; \Delta w = f(w)-  g(|\nabla w|)
$$
has an entire solution, which is then a subsolution defined on the whole space of
$$
\mathcal{M}^+(D^2w) \ge f(w)+  g(|\nabla w|)\qquad\mbox{resp. }\; \mathcal{M}^+(D^2w) \ge f(w)-  g(|\nabla w|)
$$

With these functions $v$ and $w$ at hand, we can repeat almost verbatim the proof of Proposition \ref{pro} and get a contradiction.

Theorem \ref{genteo2} is proved. \hfill $\Box$

\bigskip

\noindent {\bf Acknowledgements.}

P.F. was  partially supported by Fondecyt Grant \# 1110291, as well as
 BASAL-CMM projects, CAPDE, Anillo ACT-125, and MathAmSud project.

A. Q. was partially supported by Fondecyt Grant \#  1110210 and Programa Basal, CMM. U. de Chile  and CAPDE, Anillo ACT-125.


\end{document}